\documentclass[confernece]{IEEEtran}

\usepackage{amsfonts} 
\usepackage{amsmath}
\usepackage{amssymb}
\usepackage{amsthm}
\usepackage{xcolor}
\usepackage{color}
\usepackage{algpseudocode}
\usepackage{algorithm}
\usepackage{nicefrac}
\usepackage{comment}
\usepackage{makeidx}
\usepackage{graphicx}
\usepackage{float}
\usepackage{mathtools}
\usepackage{commath}
\usepackage{dsfont}
\usepackage{microtype}
\usepackage{subcaption}

\newcommand{\R}{\mathbb{R}}
\newcommand{\E}{\mathbb{E}}
\newcommand{\C}{\mathbb{C}}
\newcommand{\I}{\mathbb{I}}

\newcommand{\Z}{\mathbb{Z}}
\newcommand{\fft}{\mathcal{F}}
\newcommand{\ifft}{\mathcal{F}^{-1}}
\newcommand{\fftt}{\mathcal{F}_{2}}
\newcommand{\ifftt}{\mathcal{F}_{2}^{-1}}
\newcommand{\diag}{\textnormal{diag}}

\newcommand{\est}{\textnormal{est}}
\newcommand{\BC}{\textnormal{BC}}  

\newcommand{\vecb}{\mathbf{vec}}
\newcommand{\mat}{\mathbf{mat}}

\newcommand{\SO}{\mathsf{SO}}
\newcommand{\hx}{\hat{x}}
\newcommand{\hy}{\hat{y}}
\newcommand{\hrho}{\hat{\rho}}

\newcommand{\tx}{\tilde{x}}

\newcommand{\vopt}{v_{\text{opt}}}
\newcommand{\tvopt}{\tilde{v}_\text{opt}}
\newcommand{\hvopt}{\hat{v}_\text{opt}}

\newcommand{\Vopt}{V_{\text{opt}}}
\newcommand{\tVopt}{\tilde{V}_\text{opt}}

\DeclareMathOperator*{\argmax}{argmax}
\DeclareMathOperator*{\argmin}{argmin}

\newtheorem{theorem}{Theorem}[section]
\newtheorem{prop}[theorem]{Proposition}
\newtheorem{coro}[theorem]{Corollary}
\newtheorem{lemma}[theorem]{Lemma}

\numberwithin{equation}{section}

\begin{document}

\title{Provable algorithms for multi-reference alignment over $\SO(2)$}

\author{Gil Drozatz,~\IEEEmembership{}
        Tamir Bendory,~\IEEEmembership{Senior Member,~IEEE,}
        and Nir Sharon
\thanks{G. Drozatz and T. Bendory are with the School
of Electrical and Computer Engineering, Tel Aviv University, Israel. N. Sharon is with the Department of Applied Mathematics, Tel Aviv University, Israel. The research was supported in part by NSF-BSF under Grant 2019752. 
T.B. is also supported in part by BSF under Grant 2020159, in part by ISF under Grant 1924/21, and in part by a grant from The Center for AI and Data Science at Tel Aviv University (TAD). N.S. is also partially supported by the DFG award 514588180. 
}}

%
\maketitle

\begin{abstract}
	The multi-reference alignment (MRA) problem involves reconstructing a signal from multiple noisy observations, each transformed by a random group element. In this paper, we focus on the group \(\mathrm{SO}(2)\) of in-plane rotations and propose two computationally efficient algorithms with theoretical guarantees for accurate signal recovery under a non-uniform distribution over the group. 
	The first algorithm exploits the spectral properties of the second moment of the data, while the second utilizes the frequency marching principle. Both algorithms achieve the optimal estimation rate in high-noise regimes, marking a significant advancement in the development of computationally efficient and statistically optimal methods for estimation problems over groups.
\end{abstract}


\IEEEpeerreviewmaketitle

\section{Introduction}
Let \( G \) be the group of 2-D rotations, \( \SO(2) \), acting on a finite-dimensional vector space \( V \). We aim to recover a signal \( x \in V \) from \( n \) observations \( y_1, \ldots, y_n \) of the form:
\begin{equation} \label{eqn:mra}
    y_i = g_i \cdot x + \varepsilon_i, \quad i = 1, \ldots, n,
\end{equation}
where \( \cdot \) denotes the group action, \( g_i \in G \) are unknown random group elements, and $\varepsilon_i$ is a noise term. 
We assume that the group elements are drawn from a nonuniform distribution over \( \SO(2) \).
This paper studies the action of the group $\SO(2)$ on two vector spaces: bandlimited 1-D signals and bandlimited 2-D images, as detailed in Sections~\ref{sec:1D_continuous_model} and~\ref{sec:2D_continuous_model}.
We also assume that 
\( \varepsilon_i \overset{\text{i.i.d.}}{\sim} \mathcal{N}(0, \sigma^2 I) \) over the vector of coefficients in the appropriate basis of $V$ (as detailed later). 

 The studied model is a special case of the multi-reference alignment (MRA) problem, in which \( \SO(2) \) is replaced by other compact groups acting on finite-dimensional vector spaces, as introduced in \cite{bandeira2014multireference,bendory2017bispectrum,bandeira2023estimation,bendory2024transversality}.
The primary motivation for studying the MRA model is the transformative technology of single-particle cryo-electron microscopy (cryo-EM) to elucidate the spatial structure of biological molecules~\cite{scheres2012relion,cheng2018single,bendory2020single}.
Since the observations are invariant under an intrinsic group of symmetries, it is impossible to distinguish \( x \) from \( g \cdot x \) for any fixed \( g \in G \). Thus, the goal is to estimate the \( G \)-orbit of the signal: \( \{g \cdot x \mid g \in G\} \). 

To estimate the signal, we propose a two-stage framework based on the classical method of moments. First, we approximate the first two population moments using the empirical moments of the observations. This approximation is accurate if the number of observations is sufficiently large \( n \gg \sigma^4 \). 
Next, we aim to determine the parameters that represent the signal based on the estimated moments. Previous works have demonstrated that in the high-noise regime, the sample complexity of the MRA problem is governed by the lowest-order moment that uniquely determines the signal's orbit~\cite{abbe2018multireference, perry2019sample}. Since recovery from the first moment is impossible, recovery from the second moment implies that the sample complexity of the model in \eqref{eqn:mra} is proportional to \( \sigma^4 \).

The main computational challenge in the method of moments is estimating the signal from the moments. To address this, we develop two methods.
The first method employs a frequency marching approach, extending a similar technique used in discrete settings~\cite{bendory2017bispectrum}. We prove that exact recovery is possible from the first two population moments, which implies the sample complexity of the model. This proof is constructive, as it introduces an explicit computationally efficient algorithm.
The second method utilizes the spectral decomposition of an approximation of the second-moment matrix. This algorithm generalizes the spectral method used in discrete 1-D MRA~\cite{abbe2018multireference}, and we refer to it as the spectral algorithm. This algorithm \emph{approximates} the solution based on the spectral properties of the second moment matrix; see  Theorem~\ref{thm:spectral_bound_1D}.
The algorithms are detailed in Section~\ref{sec:1D_continuous_model} and Section~\ref{sec:2D_continuous_model} for the 1-D and 2-D cases, respectively. We support these findings with numerical experiments in Section~\ref{sec:Numerical_Results}.

\textbf{Contribution.} 
The uniqueness of recovering a signal from its MRA moments has been extensively studied in recent years (see, for example,~\cite{bandeira2023estimation, bendory2024transversality, perry2019sample}). While most of these works focus on uniform distributions over the group, non-uniform distributions have also been explored~\cite{abbe2018multireference, bendory2022dihedral, sharon2020method}. 
However, computationally efficient provable algorithms have been developed only for the simplest case, where the group of circular shifts acts on 1-D discrete signals~\cite{bendory2017bispectrum, abbe2018multireference, perry2019sample}. This paper introduces the first provable algorithms for a continuous group \( \SO(2) \) that acts on signals and images, marking a significant milestone in the development of provable algorithms for more complex MRA models. 
In Section~\ref{sec:cryoEM}, we discuss the potential implications for cryo-EM. This is particularly crucial since existing algorithms in cryo-EM rely on heuristics without robust validation measures.

\section{1-D MRA over $\SO$(2)} \label{sec:1D_continuous_model}
We first consider a vector space $V$ of 1-D bandlimited signals. 
Let $x(\theta)=\sum_{k=-B}^B\hx[k] e^{\iota k\theta},$ be a $B$-band-limited signal on the circle $\theta\in [0,2\pi)$,
where $\iota=\sqrt{-1}$. Here, $\hx\in\C^{2B+1}$ denotes the Fourier coefficients of $x$. 
In Fourier space, the observations are given by 
\begin{equation} \label{eqn:mra_1D}
    \hy_i[k] = g_i \cdot \hx[k] + \hat{\varepsilon}_i[k], \quad i = 1, \ldots, n,
\end{equation}
where $\hy[k]$ is the $k$-th  Fourier coefficient of the $i$-th observation,   
$g \cdot \hx[k] =  \hx[k]e^{-\iota k \phi}$, where $\phi$ is the angle associated with $g$, $\hat{\varepsilon}_i[0] \sim \mathcal{N}(0, \sigma^2)$, and for $k \geq 0$ the real and imaginary part of $\hat{\varepsilon}_i[k]$ are i.i.d.\ normal variables with zero mean and variance $\sigma^2/2$ that obey the conjugation rule $\hat{\varepsilon}_i[k] = \hat{\varepsilon}_i^*[-k]$.
As we show next, the high frequencies of the distribution are annihilated by the moments, and thus, we assume, without loss of generality, that the Fourier coefficients of the distribution of group elements, $\hrho$, is  2B-bandlimited. 

Recall that the first two population moments are given by
\begin{equation} \label{eqn:1D_Moments_def}
	\begin{split}
		M_1 &= \E\left[\hy_i\right] \in \C^{(2B+1)}, \\ 
		M_2 &= \E\left[\hy_i \hy_i^*\right] \in \C^{(2B+1)^2}.
	\end{split}
\end{equation}
The population moments can be estimated in practice by the empirical moments of the observations by averaging over the observable moments
\begin{equation}
	\begin{split}
		M_{1,\est}&\triangleq \frac{1}{n}\sum_{i=1}^{n} \hy_i, \\  M_{2,\est}&\triangleq \frac{1}{n}\sum_{i=1}^{n} \hy_i\hy^{*}_i.
	\end{split}
\end{equation}
  The population moments can be accurately estimated when $n/\sigma^4\to\infty$. This is our assumption for the rest of the paper unless stated otherwise. 

Let $T\in\C^{(2B+1)\times(2B+1)}$ be a Toeplitz matrix with elements defined by 
$T_\rho[k_1, k_2] \triangleq \hrho[k_1-k_2],$ $D_{\hx}$ is the diagonal matrix of $\hx$,  $\I$ is the identity matrix, and $\odot$ represents the Hadamard product. The moments are given explicitly in the following lemma, proven in Appendix~\ref{proof:Moment_Eq_1D}.

\begin{lemma} \label{lemma:Moment_Eq_1D}
Consider the model~\eqref{eqn:mra}. Then, 
\begin{align} \label{eqn:moments_matrix_1D}
    M_1 &= 2\pi\hx
\odot\hrho, \nonumber\\ 
M_2&=2\pi D_{\hx} T_\rho D_{\hx}^{*}+\sigma^{2}\I_{2B+1}.
\end{align}
\end{lemma}

\subsection{Frequency Marching Algorithm} \label{subsec:FM_1D}

We begin by studying a frequency marching algorithm that successively recovers the high frequencies based on the low frequencies. 
In particular, we show that given the population moments $M_1$ and $M_2$, this algorithm recovers the signal and the distribution $\rho$ uniquely. 
The underlying idea of the algorithm is to reformulate $M_2$ as a function solely of  $\hrho$, using the information of the first moment $M_1$. This allows iteratively recovering $\hrho$. Then, $\hx$ is recovered from  $M_1$ and $\hrho$. By saying that $\hx$ is \emph{non-vanishing}, we mean that all the Fourier coefficients are non-zero.

\begin{prop}\label{prop:FM_1D}
    Assume that $\hx$ and $\hrho$ are non-vanishing and  $M_{1,\est} = M_1$, and $M_{2,\est} = M_2$, Then, Algorithm~\ref{alg:FM_1D} recovers $\hx$ and $\hrho$ exactly, up to a global rotation. 
\end{prop}

\begin{algorithm} \caption{ \label{alg:FM_1D}
 A frequency marching algorithm for the 1-D model}
    \begin{algorithmic}
     \item \textbf{Input:} $M_{1,\est}$, $M_{2,\est}$, and $\sigma$
     \item \textbf{Output:} $\hx_{\est}$, $\hrho_{\est}$
    \begin{enumerate}
        \item $M_{2,\est} \gets M_{2,\est} - \sigma^2 \I_{2B+1}$  (debiasing)
        \item $S \gets 2\pi D_{M_{1,\est}}^{-1} M_{2,\est} D_{M_{1,\est}^{*}}^{-1}$
        \item $\hrho_{\est}[0] \gets \frac{1}{2 \pi}$ \label{step:FM_1D_rho_0}
        \item $\hrho_{\est}[1] \gets \sqrt{\nicefrac{1}{2 \pi S[1,1]}}$
        \For{$2 \leq k \leq B$}
        \item $\hrho_{\est}[k] \gets \label{step:2kB_FM_1D}\frac{\hrho_{\est}[1]}{S[k,k-1]\hrho_{\est}^*[k-1]} $
        \EndFor
        \For {$B+1 \leq k \leq 2B$}
        \item \label{step:Bk2B_FM_1D}
        $\hrho_{\est}[k] \gets S[k-B,-B] \hrho_{\est}[k-B] \hrho_{\est}[B]$
        \EndFor
        \For{$-2B \leq k \leq -1$}
        \item
        $\hrho_{\est}[k] = \hrho_{\est}^*[-k]$ \label{step:FM_1D_conjugate_fill}
        \EndFor
        \item $\hx_{\est} \gets \frac{M_{1,\est}}{2\pi\hrho_{\est}}$
    \end{enumerate}
    \end{algorithmic} 
\end{algorithm}

Note that the algorithm uses only a single diagonal of \(M_2\) despite the availability of many others that contain valuable information. Leveraging these additional diagonals could enhance robustness and improve overall performance. In Appendix~\ref{proof:FM_1D_all}, we prove Proposition~\ref{prop:FM_1D} and discuss this idea further.

\subsection{Spectral Algorithm} \label{subsec:Spectral_Algorithm_section_1D}
The second algorithm uses the spectral properties of $M_2$ and the close relation between circulant and Toeplitz matrices.
Let $P_{\hx}=\abs{\hx}^2$ be the power spectrum of $\hx$, which is the diagonal of the second-moment matrix.
The algorithm begins by conjugating (normalizing) the second-moment matrix by $D_{\frac{1}{\sqrt{P_{\hx}}}}$ and then extracting an isolated eigenvector that contains an approximation of the Fourier phases of the signal. This eigenvector is then combined with its Fourier magnitudes. For more details, see \cite{abbe2018multireference}.

\begin{algorithm} 
\caption{A Spectral Algorithm for the 1-D model} \label{alg:spectral_algorithm_1D}
    \begin{algorithmic}
        \item \textbf{Input:} $M_{1,\est}$, $M_{2,\est}$, and $\sigma$
        \item \textbf{Output:} $\hx_{\est}$, $\hrho_{\est}$
        \begin{enumerate}
            \item $M_{2,\est} \gets M_{2,\est} - \sigma^2 \I_{2B+1}$  (debiasing)
            
            \item $P_{\hx} \gets \diag(M_{2,\est})$
            \item $ \mathcal{M}_2 \gets D_{\frac{1}{\sqrt{P_{\hx}}}} M_{2,\est} D_{\frac{1}{\sqrt{P_{\hx}}}}$ (conjugating) \label{step:1D_mathcal_M_2}
            \item Find the eigenvalue decomposition of $\mathcal{M}_2$: eigenvalues $\lambda_0 \geq \lambda_1 \geq \ldots \geq \lambda_{2B}$ with corresponding eigenvectors $v_0, v_1, \ldots, v_{2B}$. \label{step:mathcal_m2_eigenvalue_decomp_1D}
            \item $\kappa \gets \argmax_{0 \leq k \leq 2B} \min_{k'\neq k} \abs{\lambda_{k'}-\lambda_{k}}$ \label{step:unique_eigenvector_1D}
            \item $\tx_{\est} \gets \sqrt{2B+1} v_{\kappa}$ 
            \item $\beta \gets e^{\iota (\measuredangle M_{1,\est}[0] - \measuredangle \tx_{\est}[0])}$ 
            \item $\tx_{\est} \gets \beta \tx_{\est}$ \label{step:def_x_tilde_1D}
            \item $\hx_{\est} \gets  \sqrt{P_{\hx}}\odot \tx_{\est}$
            \item $\hrho_{\est} \gets \frac{M_{1,\est}}{2 \pi \hx}$
        \end{enumerate}
    \end{algorithmic}
\end{algorithm}

The crux of the spectral algorithm is the similarity between Toeplitz and circulant matrices and the projection of a given Hermitian Toeplitz matrix into the space of circulant matrices, which are diagonalized by the DFT matrix. Consider the problem:
\begin{equation}
\argmin_{v\in\C^{2B+1}} \norm{T_\rho-C_{v}}^{2}_{F},
\end{equation}
where $v$ is a vector defining the circulant matrix $C_{v}$. This problem enjoys a closed-form solution~\cite{chan1988circulant_preconditioner}, given by 
\begin{equation} \label{eqn:v_opt}
    v_{\text{opt}}[k]=\frac{k\hrho\left[-(2B+1-k)\right]+(2B+1-k)\hrho[k]}{2B+1},
\end{equation}
and the distance between $T_\rho$ and the associated circulant matrix (the error) is given by 
\begin{equation} \label{eqn:T_minus_C_min}
    \resizebox{\linewidth}{!}{$S_B(\hrho) = \sum_{k=1}^{2B}\abs{\hrho[k]-\hrho[-(2B+1-k)]}^2 \left(\frac{k(2B+1-k)}{2B+1}\right).$}
\end{equation}

The following theorem shows that given the ideal assumption of $T_{\rho}$ being a circulant matrix, Algorithm~\ref{alg:spectral_algorithm_1D} recovers the signal and the first $B$ Fourier coefficients of $\rho$.

\begin{prop}
 \label{prop:spectral_algorithm_perfect_1D}
    Assume that  $S_B(\hrho)=0$, $\hx$ is non-vanishing and that $\hvopt$ has an isolated entry. In addition, assume that  $M_{1,\est}=M_1$, $M_{2,\est}=M_2$. Then, Algorithm~\ref{alg:spectral_algorithm_1D} recovers $\hx$ and the first B Fourier coefficients of $\rho$ exactly, up to a global rotation.
  \end{prop} 
  
The proof of Proposition~\ref{prop:spectral_algorithm_perfect_1D} can be found in Appendix~\ref{proof:spectral_algorithm_perfect_1D}.
In practice, we do not expect $T_\rho$ to be a circulant matrix, implying $S_B(\hrho)>0$. 
Hence, in practice, the spectral algorithm only approximates the solution, and the estimation error depends on $S_B(\hrho)$.   
This is captured by the following result, based on the Davis-Kahan Theorem~\cite{davis_kahan1970sin_theta}. 

\begin{theorem} \label{thm:spectral_bound_1D}
Consider $\lambda^{T}_0 \geq \lambda^{T}_1 \geq \ldots \geq \lambda^{T}_{2B}$ and $\lambda^{C}_0 \geq \lambda^{C}_1 \geq \ldots \geq \lambda^{C}_{2B}$ to be the eigenvalues of $T_\rho$ and $C_{\vopt}$, respectively. We denote $P_{\max}:=\max_{0\leq k \leq B}P_{\hx}[k]$ and $\delta_{\kappa}=\max\left(\min_{j \neq \kappa}{\abs{\lambda^{C}_{\kappa}-\lambda^{T}_{j}}}, \min_{j \neq \kappa}{\abs{\lambda^{C}_{j}-\lambda^{T}_{\kappa}}}\right)$.
   Assume the following conditions hold:
   \begin{enumerate}
       \item 
       $\hx$ is non-vanishing.
       \item 
       $\kappa$ of Step~\ref{step:unique_eigenvector_1D} of Algorithm~\ref{alg:spectral_algorithm_1D} has both $\lambda_{\kappa}^{T}$ and $\lambda_{\kappa}^{C}$ with an eigenspace of dimension 1.
       \item 
       The estimate calculated in Step~\ref{step:def_x_tilde_1D} of Algorithm~\ref{alg:spectral_algorithm_1D} satisfies  $\tx_{\est}^*[\Phi_{\frac{2\pi l}{2B+1}} \odot\tx] \geq 0$.
       \item 
       $S_B(\hrho) \leq \delta_{\kappa}^2$.
   \end{enumerate} 
   Then, there exists $ 0 \leq l \leq 2B$, such that
    \begin{equation}
    \begin{split}
    \resizebox{\linewidth}{!}{$\norm{\hx_{\est} - \Phi_{\frac{2\pi l}{2B+1}} \odot\hx}_F^2 \leq 2(2B+1)P_{\max}
    \left[1- \sqrt{1-\frac{S_B(\hrho)}{\delta_{\kappa}^{2}}}\right],$}
    \end{split}
    \end{equation}
    where $\Phi_{\varphi}[k] = e^{-\iota k \varphi}$.
     \end{theorem}
     
The proof of Theorem~\ref{thm:spectral_bound_1D} is in Appendix~\ref{proof:spectral_bound_1D}.
\section{2-D MRA over $\SO$(2)} \label{sec:2D_continuous_model}
We now extend the algorithms from 1-D signals to 2-D images. 
We assume that $V$ is the space of bandlimited images in the sense that they can be presented with finitely many Fourier-Bessel coefficients~\cite{zhao2016fast}. 
Namely, the sought image is of the form
\begin{equation*}
\resizebox{\linewidth}{!}{$
x\left(\theta,r\right)=\sum_{\left(k,q\right)\in I}\hx[k,q] e^{\iota k\theta} J_{q}\left(r\right), \quad \theta \in[0,2\pi), \quad r \in [0,1], $}
\end{equation*}
where  $u_{k,q}(\theta,r) = e^{\iota k\theta} J_{q}\left(r\right)$ is the general Fourier-Bessel function, with $J_q(r)$ being the cylindrical Bessel function of order $q$. The set $\{u_{k,q}(\theta,r)\}_{k\in \Z,q \in \Z_{\geq 0}}$ spans all the ``nice enough'' functions over the unit disc. 
The images are bandlimited in the sense that there exists an angular bandwidth $B$ and a maximal radial frequency $Q_{k}$ such that the set  $I=\left\{ (k,q) \colon -B\leq k\leq B,\, 0\leq q\leq Q_{k}-1\right\}$ is finite. We denote the total number of coefficients by $|I|$.   
In Appendix~\ref{proof:rotations_2D}, we prove that
\begin{equation} \label{eqn:rotation_calc_2D}
    x(\theta-\varphi,r) = \sum_{\left(k,q\right)\in I}\hx[k,q]e^{-\iota  k\varphi}u_{k,q}\left(\theta,r\right).
\end{equation}
Thus, in Fourier-Bessel space, the observations are given by 
\begin{equation} \label{eqn:mra_2D}
	\hy_i[k,q] = g_i \cdot \hx[k,q] + \hat{\varepsilon}_i[k,q], \quad i = 1, \ldots, n,
\end{equation}
where $\hy[k,q]$ is the $(k,q)$-th  Fourier-Bessel coefficient of the $i$-th observation,   
$g \cdot \hx[k,q] =  \hx[k,q]e^{-\iota k \phi}$, where $\phi$ is the angle associated with $g$, $\hat{\varepsilon}_i[0,q] \sim \mathcal{N}(0, \sigma^2)$, and for $k \geq 0$ the real and imaginary part of $\hat{\varepsilon}_i[k,q]$ are i.i.d.\ normal variables with zero mean and variance $\sigma^2/2$ that obey the conjugation $\hat{\varepsilon}_i[k, q] = \hat{\varepsilon}_i^*[-k, q]$.
In this sequel, we order the Fourier-Bessel coefficients in lexicographical order and think of them as a column vector with $|I|$ complex entries.

The following lemma gives the first two moments. It is proven in Appendix~\ref{proof:Moment_Eq_2D}.
\begin{lemma} \label{lemma:Moment_Eq_2D}
	Consider the 2-D MRA problem~\eqref{eqn:mra_2D}.
	The first two moments are given by 
	\begin{align}
		M_1&=2\pi\hx\odot
		\mathcal{R},\\
		M_2&=2\pi D_{\hx}\mathcal{T}_{\rho} D_{\hx^*} + \sigma^2 \I_{|I|}.
	\end{align} 
	Here, $\mathcal{R}\in \C^{|I|}$ and 
	$\mathcal{T}_{\rho} \in \C^{|I|\times|I|}$ are block matrices with blocks indexed by $-B \leq k_1,k_2 \leq B$ and 
are defined as:
\begin{align*}
	(\mathcal{T}_{\rho})_{k_1 k_2}&= \hrho[k_1-k_2]\mathds{1}_{Q_{\abs{k_1}}\times Q_{\abs{k_2}}},\\
	(\mathcal{R})_{k_1} &= \hrho[k_1]\mathds{1}_{Q_{\abs{k_1}}\times 1},
\end{align*}
where 	$\mathds{1}_{i\times j}$ denotes a matrix of ones of size $i \times j$.	
\end{lemma}

\subsection{Frequency Marching Algorithm} \label{subsec:FM_2D}
We consider a frequency marching algorithm, similar to Algorithm \ref{alg:FM_1D}. The difference stems from the addition of the radial coordinate that provides more information and, thus, stability. Proposition~\ref{prop:FM_2D} is proven in Appendix~\ref{proof:FM_2D}.

\begin{prop} \label{prop:FM_2D}
     Assuming $M_{1,\text{est}}=M_1$, $M_{2,\text{est}}=M_2$, there exists a frequency marching algorithm, analog to Algorithm~\ref{alg:FM_1D}, that recovers $\hx$ and $\hrho$ exactly, up to a global rotation. 
\end{prop} 
  
\subsection{Spectral Algorithm} \label{subsec:Spectral_Algorithm_section_2D}

Similarly to the 1-D case, a spectral algorithm can be designed for the 2-D case, analogous to Algorithm~\ref{alg:spectral_algorithm_1D}. The main difference is that the Toeplitz and circulant matrices are replaced by block Toeplitz and block circulant matrices, respectively.  
As in the 1-D case, if the block Toeplitz matrix happens to be a block circulant matrix, the spectral algorithm can recover the signal—up to a global rotation—using second-moment information. Otherwise, the reconstruction error is proportional to the distance between the block Toeplitz matrix and its associated circulant approximation. In Appendix~\ref{prop:spectral_algorithm_perfect_2D} we design a 2-D spectral algorithm, and bound its reconstruction error in Appendix~\ref{thm:spectral_bound_2D}.

\begin{figure*}[ht]
    \centering    
    \subfloat[]{
\includegraphics[width=.3\linewidth]{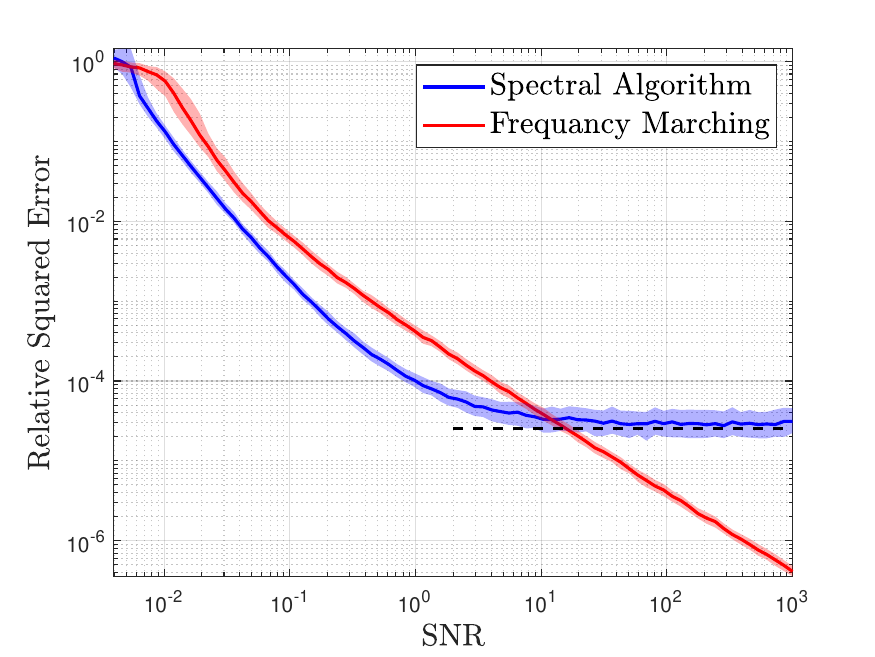}

\label{fig:N1}
    }
  \subfloat[]{
\includegraphics[width=.3\linewidth]{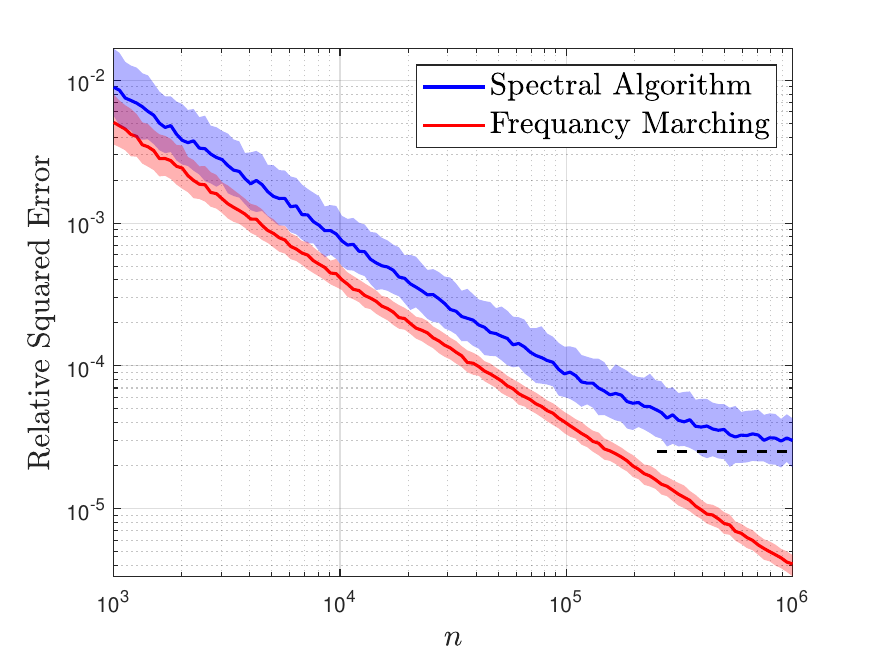}

\label{fig:func_of_N_median}
    }    
\subfloat[]{
    \includegraphics[width=.3\linewidth]{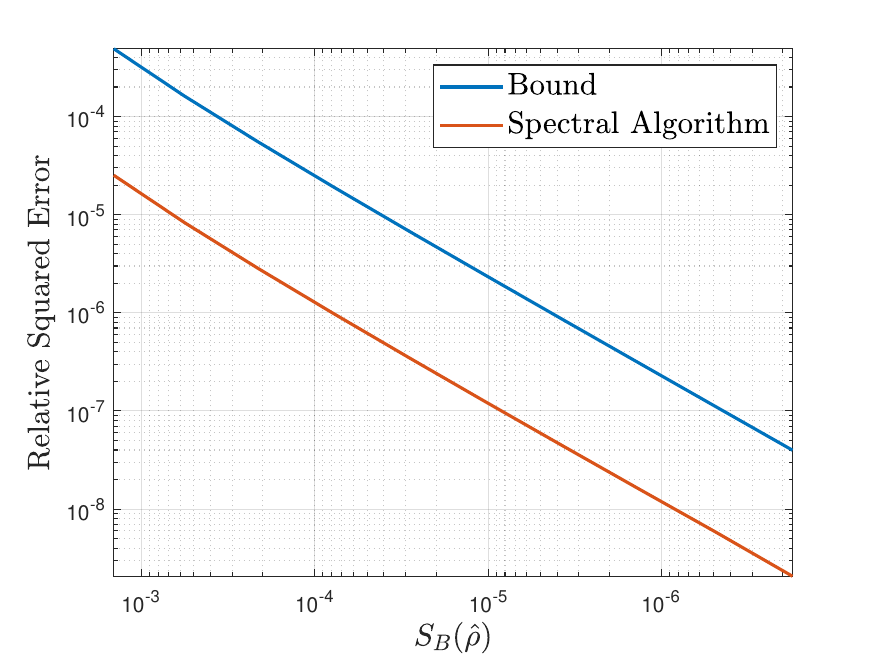}
    \label{fig:bound_with_2_errors}
    }
    \caption{
    	(a) Recovery error of the frequency marching and spectral algorithms, with 20\% error margins, as a function of the SNR for \( n = 10^6 \) observations.  
    	(b) Recovery error of the frequency marching and spectral algorithms, with 20\% error margins, as a function of the number of observations for \(\text{SNR} = 100\).  
    	(c) Recovery error of the spectral algorithm as a function of \(S_B(\hrho)\), compared to the theoretical bound from Theorem~\ref{thm:spectral_bound_1D}.  
    	}
    \label{fig:experiments}
\end{figure*}

\section{Numerical Results} \label{sec:Numerical_Results}
We present the numerical results of both algorithms for the 2-D case. The performance for the 1-D case is similar, but due to space constraints, it is omitted.  
In all experiments, the ground truth image was generated with \( B = 10 \) and \( Q = 2 \). The magnitudes of all Fourier–Bessel coefficients were set to one, and their phases were drawn from a uniform distribution, subject to the constraint that the image remains real.  
The real and imaginary components of the distribution were initially drawn from a uniform distribution over the interval \([0, 1]\) and subsequently corrected according to Equation~\eqref{eqn:v_opt} to ensure \(S_B(\hrho) = 0\) while preserving both the positivity and normalization of the distribution in the signal domain (as it is a distribution).  
Next, we perturbed the distribution as
\begin{equation}\label{eqn:perturbation}
\hrho[k] \to e^{\iota \eta \sqrt{k}}\hrho[k],
\end{equation}
with \(\eta = 0.1\). In the experiments presented in Figure~\ref{fig:experiments}a and~\ref{fig:experiments}b, this perturbation resulted in \(S_B(\hrho) = 0.0014\). We define the recovery error as
\begin{equation}
    \frac{\min_{\varphi\in[0,2\pi)}{\norm{\hx_{\est} - e^{-\iota k \varphi} \odot\hx}_F^2}}{\norm{\hx}_F^2}.
\end{equation}

In the experiment whose results are shown in Figure~\ref{fig:experiments}a, we present the recovery error as a function of the SNR, defined by  
\[
\text{SNR} = \frac{\sum_{k,q}P_{\hx}[k,q]}{(2B+1) Q \sigma^2},
\]  
with \( n = 10^6 \) observations. We report the median error with 20\% error margins, computed over 400 trials per SNR value.  
As expected, the error decreases as the SNR increases. For higher SNR values, the error of the spectral algorithm stagnates because \( S_B(\hrho) \neq 0 \) (see Theorem~\ref{thm:spectral_bound_1D}). In contrast, the error of the frequency marching algorithm approaches zero at high SNR, as predicted by Proposition~\ref{prop:FM_2D}. The error rate of the frequency marching is proportional to $1/\text{SNR}$.  
However, at lower SNR values, the spectral algorithm outperforms the frequency marching algorithm. This is because the spectral algorithm processes all second-moment information simultaneously, while the frequency marching algorithm, due to its sequential nature, accumulates errors over frequencies.

The experiment shown in Figure~\ref{fig:experiments}b presents the recovery error as a function of the number of observations for \(\text{SNR} = 100\). We report the median error over 800 trials with 20\% error margins.  
In this high-SNR regime, the frequency marching algorithm outperforms the spectral algorithm, achieving an error rate proportional to \(1/n\).  

In the experiment presented in Figure~\ref{fig:experiments}c, we assume access to the exact moments and compare the error of the spectral algorithm with the theoretical bound from Theorem~\ref{thm:spectral_bound_1D}. This comparison is performed by varying the parameter \(\eta\) in~\eqref{eqn:perturbation} from \(10^{-1}\) to \(10^{-3}\) and optimizing over all possible rotations of $\rho$ to achieve the minimal bound.  
The theoretical bound closely follows the numerical results of the spectral algorithm as a function of \(S_B(\hrho)\). However, a noticeable gap remains, indicating that Theorem~\ref{thm:spectral_bound_1D} is not tight.  

\section{Vision: Provable algorithms for cryo-EM} \label{sec:cryoEM}
This paper introduces provable frequency marching and spectral algorithms for the MRA problem with the group of in-plane rotations, \(\mathrm{SO}(2)\). This represents a significant advancement toward developing rigorous algorithms for more complex MRA models.

Our primary motivation is cryo-EM, which can be modeled as an MRA model with the group of 3-D rotations $\SO(3)$ and an additional linear operator (which is not modeled by~\eqref{eqn:mra})~\cite{bendory2020single}.
Despite being widely used and occasionally yielding excellent results, existing cryo-EM algorithms lack theoretical guarantees and robust validation measures. Moreover, cryo-EM algorithms are susceptible to well-documented pitfalls and model biases, such as the ``Einstein from noise'' phenomenon~\cite{henderson2013avoiding,balanov2024einstein}.
To address these challenges, our vision is to design provable algorithms for cryo-EM—whether through frequency marching, spectral methods, or alternative techniques—to enhance confidence in cryo-EM reconstructions and ultimately drive new biological discoveries.

\appendices
\section{Proof of Lemma~\ref{lemma:Moment_Eq_1D}}
\label{proof:Moment_Eq_1D}
\begin{proof}
We know that $\hy_{j}[k] = \hx[k]e^{-\iota k\varphi_j}+\hat{\varepsilon}[k]$ for $-B \leq k \leq B$. Substituting this in (\ref{eqn:1D_Moments_def}):
\begin{equation}
\begin{split}
    M_1[k]&=\E\left[\hx[k]e^{-\iota k\varphi}+\hat{\varepsilon}[k]\right]\\
    &=\hx[k]\E\left[e^{-\iota k \varphi}\right]+\E\left[\hat{\varepsilon}[k]\right]= \hx[k]\E\left[e^{-\iota k \varphi}\right] ,
\end{split}
\end{equation}
and
\begin{equation}
\begin{split}
    M_{2}\left[k_{1},k_{2}\right]&=
    \mathbb{E}\left[\hx[k_1]\hx[k_2]^{*}e^{-\iota k_{1}\varphi}e^{\iota k_{2}\varphi}\right] \\ &+\mathbb{E}\left[\hat{\varepsilon}[k_{1}]\hat{\varepsilon}[k_{2}]^{*}\right] \\
    &=\hx[k_1]\hx[k_2]^{*}\mathbb{E}\left[e^{-\iota\left(k_{1}-k_{2}\right)\varphi}\right]+\sigma^{2} \delta_{k_1,k_2}.
\end{split}
\end{equation}
Here $\delta_{\cdot,\cdot}$ is the Kronecker delta. Note that $\E\left[e^{-\iota k\varphi}\right]$ is 
\begin{equation}
\E\left[e^{-\iota k\varphi}\right]=\int_{0}^{2\pi}e^{-\iota k\theta}\rho\left(\theta\right)d\theta=2\pi\hrho[k].
\label{fourier_of_rho}
\end{equation}
Therefore,
\begin{align}
    M_1[k]&=2\pi \hx[k]\hrho[k] , \\
    M_{2}[k_{1},k_{2}]&=2\pi \hx[k_1]\hx[k_2]^{*}\hrho[k_{1}-k_{2}]+\sigma^{2}\delta_{k_1,k_2},
\end{align}
concluding the proof.
\end{proof}

\section{On the 1-D Frequency Marching Algorithm} \label{proof:FM_1D_all}
\subsection{Proof of Proposition~\ref{prop:FM_1D}} 
\label{proof:FM_1D}
\begin{proof}
     First, notice that \(D_{M_1}=2\pi D_{\hx}D_{\hrho}\), so \(D_{\hx}=\frac{1}{2\pi}D_{M_1}D_{\hrho}^{-1}\). The de-biased $M_2$ can therefore be re-written as
    \begin{equation}
    M_2=\frac{1}{2\pi}D_{M_1}D_{\hrho}^{-1} T_{\rho} D_{\hrho^{*}}^{-1}D_{M_1^{*}}.
    \end{equation}
    This means that conjugating $M_2$ by $D_{M_1}^{-1}$ removes the dependency on $\hx$. Thus, $S$ is 
    \begin{equation} \label{S_def}
        S \triangleq 2\pi D_{M_1}^{-1} M_2 D_{M_1^{*}}^{-1}=D_{\hrho}^{-1} T_{\rho} D_{\hrho^{*}}^{-1} .
    \end{equation}
    Namely,
    \[ S[k_1,k_2] = \frac{\hrho[k_1-k_2]}{\hrho[k_1] \hrho^*[k_2]}, \: -B \leq k_1,k_2 \leq B .\]
    In particular, on the diagonal $ S[k,k] = \frac{\hrho[0]}{\hrho[k] \hrho^*[k]} = \frac{\hrho[0]}{\left|\hrho[k]\right|^{2}}$.
    In addition, since $\rho$ is a probability distribution, we have $$\hrho[0] = \frac{1}{2\pi}\int_{0}^{2\pi}\rho(\theta) d\theta=\frac{1}{2\pi}.$$ 
    
    We fix $\hrho_{\est}[1]$ to have a unit phase, $\hrho_{\est}[1]=|\hrho[1]| = \sqrt{\frac{1}{2 \pi S[1,1]}}$. Then, we find $\left\{\hrho[k]\right\}_{k=1}^{2B}$ in ascending order, using all previously calculated $\hrho[k]$'s and elements of $S$, via the following scheme; for any $ 2 \leq k \leq B$:
    \begin{equation}
        \hrho_{\est}[k] = \frac{\hrho_{\est}[1]}{S[k,k-1]\hrho_{\est}^*[k-1]}
    \end{equation}
    and for any $ B+1 \leq k \leq 2B$:
    \begin{equation}
        \begin{split}
        \hrho_{\est}[k] &= S[k-B,-B] \hrho_{\est}[k-B]\hrho_{\est}^*[-B]\\
        &= S[k-B,-B] \hrho_{\est}[k-B]\hrho_{\est}[B].
        \end{split}
    \end{equation}
    This process extracts $\hrho[k]$ from $S$ and determines $1 \leq k \leq 2B$, under the assumption that $\hrho_{\est}[1]$ is equal to $\hrho[1]$. The negative frequencies are decided by the realness of $\rho(\theta)$. Note that the above assumption reflects the freedom to choose an arbitrary rotation.
\end{proof}

\subsection{A robust frequency marching algorithm} \label{subsec:better_FM_1D}
Algorithm~\ref{alg:FM_1D} uses just one entry of $S$ for each frequency $k$, which is somewhat arbitrary, since for $0 \leq k \leq B$ we have $$\hat{{\rho}}[k] = \frac{\hrho[k-k']}{S[k,k']\hrho[k']^*}, \quad 1 \leq k'\leq k-1 .$$
For $B+1 \leq k \leq 2B$, we have 
$$\hrho[k] = S[k-k',-k']\hrho[k-k']\hrho[k'], \quad k-B \leq k'\leq B .$$ We next briefly present several modifications that define an improved and robust alternative algorithm, particularly relevant for practical applications involving the empirical moments \( M_{1,\text{est}} \) and \( M_{2,\text{est}} \).

Define two sets of weights $\left\{\omega_{k;k'}\right\}_{k'=1}^{k-1}$ and $\left\{\tilde{\omega}_{k;k'}\right\} _{k'=k-B}^{B}$, where  $\sum_{k'=1}^{k-1}\omega_{k;k'}=\sum_{k'=k-B}^{B}\tilde{\omega}_{k;k'}=1$. We use the diagonal of $S$ in each marching step to mitigate cascading errors and only for frequencies $0 \leq k \leq B$, where $S[k,k]$ exists. In particular, Step~\ref{step:2kB_FM_1D} takes the following form:
\begin{align}
    \tilde{\rho}_{\est}[k]&\triangleq \sum_{k'=1}^{k-1}\omega_{k;k'}\frac{\hrho_{\est}[k-k']}{S[k,k']\hrho_{\est}[k']^*} \, ,\\
    \hrho_{\est}[k]&=\frac{1}{\sqrt{2\pi S[k,k]}}\frac{\tilde{\rho}_{\est}[k]}{\left|\tilde{\rho}_{\est}[k]\right|} \, .
\end{align}
Similarly, Step~\ref{step:Bk2B_FM_1D} becomes for $B+1\leq k \leq 2B$:
\begin{equation}
    \hrho_{\est}[k]=\sum_{k'=k-B}^{B}\tilde{\omega}_{k;k'}S[k-k',-k']\hrho_{\est}[k-k']\hrho_{\est}[k'].
\end{equation}
$\hx$ is again recovered in the same way as before.

\section{Proof of the 1-D Spectral Algorithm}
\label{proof:spectral_algorithm_perfect_1D}
\begin{proof}
    Let $W\in\C^{(2B+1)^2}$ be the orthogonal DFT matrix. This proof uses the fact that any circulant matrix $C_v$ is diagonalized by $W$, i.e.,
    \begin{equation}
        C_v = W^* D_{\fft(v)} W .
    \end{equation}
    
    Consider $\mathcal{M}_2$ and rewrite Step~\ref{step:1D_mathcal_M_2}:
    \begin{equation} \label{eqn:eigen_decomp_1D}
    \begin{split}
        \mathcal{M}_2 &\triangleq D_{\frac{1}{\sqrt{P_{\hx}}}} M_2 D_{\frac{1}{\sqrt{P_{\hx}}}}\\
        &=  2\pi D_{\frac{1}{\sqrt{P_{\hx}}}}D_{\hx} C_{\vopt} D_{\hx}^{*}D_{\frac{1}{\sqrt{P_{\hx}}}}\\  
        &= 2\pi D_{\tx} C_{\vopt}  D_{\tx}^{*}\\
        &=2\pi WC_{\ifft(\tx)}W^{*} C_{\vopt} WC_{\ifft(\tx)}^{*}W^{*},
    \end{split}
    \end{equation}
    where $\ifft$ denotes the inverse DFT.
    Since $W^{*} C_{\vopt} W = \left(W C_{\vopt}^t W^{*}\right)^{t}$,  we have that $C_{\vopt}^t = C_{\tvopt}$ is a hermitian circulant matrix, defined by its first column,
    \begin{equation} \label{eqn:v_tilde_opt}
        \tvopt[k] = \vopt[-k \bmod 2B+1], \quad 0 \leq k \leq 2B.
    \end{equation}
It is known that 
    \begin{equation} \label{eqn:tilde_vopt_fft}
        \fft(\tvopt)[b] = \hvopt[-b \bmod 2B+1], \quad  0\leq b \leq 2B.
    \end{equation}
    This means that
    \begin{equation} \label{eqn:mathcal_m_2_eigenvalue_decomp}
        \mathcal{M}_2 = 2\pi  W C_{\ifft(\tx)} D_{\fft(\tvopt)}C_{\ifft(\tx)}^{*} W^*.
    \end{equation}
    In addition, $C_{\ifft(\tx)}$ is a unitary matrix, as
    \begin{equation*}
    \begin{split}
C_{\ifft(\tx)}C_{\ifft(\tx)}^{*}&=W^*WC_{\ifft(\tx)}W^*WC_{\ifft(\tx)}^{*}W^*W\\
        &=W^*D_{\tx}D_{\tx}^*W 
        =W^*D_{\frac{|\hx|^2}{P_{\hx}}}W=\I_{2B+1}.
    \end{split}
    \end{equation*}
We can, therefore, conclude that~\eqref{eqn:mathcal_m_2_eigenvalue_decomp} is an orthogonal eigenvalue decomposition, where the columns of $WC_{\ifft(\tx)}$ are the eigenvectors. 

In steps~\ref{step:mathcal_m2_eigenvalue_decomp_1D} and~\ref{step:unique_eigenvector_1D}, we find a normalized eigenvector $v_{\kappa}$ of $\mathcal{M}_2$ corresponding to a non-degenerate eigenvalue, known to exist from the assumption of $\hvopt$ having an isolated entry. The non-degenerate eigenvector $v_{\kappa}$ must then be equal to a column of $WC_{\ifft(\tx)}$ up to some global phase, i.e.,
\begin{equation} \label{eqn:fft1_l_0_2B}
    v_{\kappa} = \alpha W \left( l\circ\ifft(\tx) \right) = \frac{\alpha}{\sqrt{2B+1}} \fft \left( l\circ\ifft(\tx) \right),
\end{equation}
for some $0 \leq l \leq 2B$ and $\abs{\alpha}=1$, where $\circ$ represents a circular shift of a vector. This means that $v_{\kappa}$ is the same as $\tx$, up to normalization by $\sqrt{2B+1}$ and multiplication by a phase. In the image, it means a rotation by $\varphi=\frac{2\pi l}{2B+1}$ and a possible global phase. Finally, $\hrho$ is recovered through $M_1$.
\end{proof}

\section{Proof of Bound on the 1-D Spectral Algorithm}
\label{proof:spectral_bound_1D}
\begin{proof}[Proof of Theorem~\ref{thm:spectral_bound_1D}]
By Appendix~\ref{proof:spectral_algorithm_perfect_1D}, there exists $l \in \{0,\ldots,2B\}$ s.t. the vector
    \begin{equation}
        v_{\kappa}^C = \frac{1}{\sqrt{2B+1}}\Phi_{\frac{2\pi l}{2B+1}} \odot \tx,
    \end{equation}
    is the unique normalized eigenvector of 
    \begin{equation}
        \mathcal{M}_{2}^{C} \triangleq 2\pi W C_{\ifft(\tx)}W^{*} C_{\vopt} W C_{\ifft(\tx)}^{*} W^*,
    \end{equation}
    with eigenvalue $2\pi\lambda_{\kappa}^C$, where we recall that $2\pi \lambda^T$ and $2\pi \lambda^C$ correspond to eigenvalues of $\mathcal{M}_2$ and $\mathcal{M}_2^C$, which differ from $T_{\rho}$, $C_{\vopt}$ only by conjugation by the same unitary matrix and a constant. Then,
    \begin{equation} \label{eqn:1D_error_bound_1st_part}
    \begin{split}
        &\norm{\hx_{\est}-\Phi_{\frac{2\pi l}{2B+1}}\odot \hx}^2_F \\
         = & \norm{\sqrt{P_{\hx}} \odot \left[ \tx_{\est} - \sqrt{2B+1} v_{\kappa}^C\right]}^2_F\\
         \leq & \left[\max_{0\leq k \leq B} P_{\hx}[k]\right] \norm{\tx_{\est} - \sqrt{2B+1} v_{\kappa}^C}^2_F\\
          = & (2B+1) \left[\max_{0\leq k \leq B} P_{\hx}[k]\right] \norm{\beta v_{\kappa}^T- v_{\kappa}^C}^2_F.\\
    \end{split}
    \end{equation}
    
    To further bound $\norm{\beta v_{\kappa}^T- v_{\kappa}^C}^2_F$, we are going to require the following lemmas.
    \begin{lemma} \label{lem:a_est_conj_flip_1D}
        $\tx_{\est}$ has the following property:
        \begin{equation}
            \tx_{\est}[k] = \tx_{\est}^*[-k].
        \end{equation}
    \end{lemma}
    \begin{proof}
    
    Let $v$ be some eigenvector of $\mathcal{M}_2 = 2\pi D_{\tx} T_{\rho} D_{\tx}^{*}$ with eigenvalue $\lambda$. Consider the vector $v'$ defined by $v'[k] = v^*[-k]$:
    \begin{equation}
    \begin{split}
        &\left[2\pi D_{\tx} T_{\rho} D_{\tx}^{*} v'\right][k_1] = \\
        & \sum_{k_2=-B}^{B} 2\pi \tx[k_1] \tx[k_2]^{*} \hrho[k_1-k_2] v^{*}[-k_2]=\\
        & \sum_{k_2'=-B}^{B} 2\pi \tx[k_1] \tx[-k_2']^{*} \hrho[k_1+k_2'] v^{*}[k_2'] =\\
        & \left[\sum_{k_2'=-B}^{B} 2\pi \tx[-k_1] \tx[k_2']^{*} \hrho[-k_1-k_2'] v[k_2']\right]^{*} =\\
        &\left[\lambda v[-k_1]\right]^{*} = \lambda v'[k_1].\\
    \end{split}
    \end{equation}
    Namely, $v'$ is an eigenvector of $\mathcal{M}_2$ with eigenvalue $\lambda$. Since $\tx_{\est}$ is an eigenvector of $\mathcal{M}_2$ corresponding to an isolated eigenvalue:
    \begin{equation}
        \tx_{\est}[k] = \alpha \tx_{\est}^*[-k]
    \end{equation}
    for some $\alpha\in\C$. The realness of $\tx_{\est}[0]\neq 0$ forces $\alpha =1$, which completes our proof.
    \end{proof}
    \begin{lemma} \label{lemma:conj_flip_eq_1D}
        Let $v,u\in \C^{2B+1}$, s.t. $v[k] = v^*[-k]$ and $u[k] = u^*[-k]$. Then, $u^*v\in\R$.
        \end{lemma}
    \begin{proof}
        Under the lemma's condition, the two functions $$f_{u}(\theta)=\sum_{k=-B}^B u[k] e^{\iota k \theta} \text{ and } f_{v}(\theta)=\sum_{k=-B}^B v[k] e^{\iota k \theta},$$ are  real. Using Parseval's theorem, we get:
        \begin{equation}
            u^*v = \frac{1}{2\pi} \int_{0}^{2 \pi} f^*_{u}(\theta) f_{v}(\theta) \in \R.
        \end{equation}
    \end{proof}
    To further bound~\eqref{eqn:1D_error_bound_1st_part}, we note that $v_{\kappa}^C [k] = v_{\kappa}^{C*} [-k]$. Then, 
    \begin{equation}
    \begin{split}
       &\norm{\beta v_{\kappa}^T- v_{\kappa}^C}^2_F =\\
       & \norm{\beta v_{\kappa}^T}^2_F + \norm{v_{\kappa}^C}^2_F - 2\Re\left[ \left(\beta v_{\kappa}^T\right)^* v_{\kappa}^C \right] = \\
       &2 - 2\Re\left[ \left(\beta v_{\kappa}^T\right)^* v_{\kappa}^C \right] \stackrel{(*)}{=} 2 - 2\left[ \left(\beta v_{\kappa}^T\right)^* v_{\kappa}^C \right].
    \end{split}
    \end{equation}
    Where at (*) we employ both Lemma~\ref{lem:a_est_conj_flip_1D} and Lemma~\ref{lemma:conj_flip_eq_1D}. At this point we use that $\left(\beta v_{\kappa}^T\right)^* v_{\kappa}^C \geq 0$, guaranteed by the conditions of the theorem, to have
    \begin{equation} \label{eqn:fft_bound_1D}
    \begin{split}
        \norm{\beta v_{\kappa}^T- v_{\kappa}^C}^2_F &= 2 - 2\abs{ \left(\beta v_{\kappa}^T\right)^* v_{\kappa}^C } \\
        &= 2 \left[ 1-\abs{v_{\kappa}^{T*} v_{\kappa}^C} \right].
    \end{split}
    \end{equation}
    All we are now left with is finding lower bound on $\abs{v_{\kappa}^{T*} v_{\kappa}^C}$. To that end, we will use the Davis-Kahan $\sin\theta$ theorem, which states that
    \begin{equation} \label{eqn:davis_kahan}
        \abs{\sin\measuredangle(v^{C}_{\kappa},v^{T}_{\kappa})} \leq \frac{\norm{\mathcal{M}_{2}-\mathcal{M}_{2}^{C}}_{F}}{2\pi \delta_{\kappa}},
    \end{equation}
    with 
    \[ \delta_{\kappa}=\max\left(\min_{j \neq \kappa}{\abs{\lambda^{C}_{\kappa}-\lambda^{T}_{j}}}, \min_{j \neq \kappa}{\abs{\lambda^{C}_{j}-\lambda^{T}_{\kappa}}}\right).\]
    Since $\abs{\cos\measuredangle(v^{C}_{\kappa},v^{T}_{\kappa})} = \abs{v^{C*}_{\kappa}v^{T}_{\kappa}}$, we have that 
    \[ 
    \begin{split}
    \abs{\sin\measuredangle(v^{C}_{\kappa},v^{T}_{\kappa})} & = \sqrt{1-\cos^{2}\measuredangle(v^{C}_{\kappa},v^{T}_{\kappa})} \\
    &= \sqrt{1-\abs{v^{C*}_{\kappa}v^{T}_{\kappa}}^{2}}.
    \end{split}
    \] 
    Since $\frac{S_B(\hrho)}{\delta_{\kappa}^{2}} \leq 1$,
    \begin{equation} \label{eqn:v_C_v_T_bound_1D}
    \begin{split}
        \sqrt{1-\abs{v^{C*}_{\kappa}v^{T}_{\kappa}}^{2}} &\leq \frac{2\pi \norm{T_{\rho}-C_{\vopt}}_{F}}{2\pi \delta_{\kappa}}\\
        1-\abs{v^{C*}_{\kappa}v^{T}_{\kappa}}^{2} &\leq \frac{S_B(\hrho)}{\delta_{\kappa}^{2}}\\
        \sqrt{1-\frac{S_B(\hrho)}{\delta_{\kappa}^{2}}} &\leq \abs{v^{C*}_{\kappa}v^{T}_{\kappa}}.
    \end{split}
    \end{equation}
    After combining~\eqref{eqn:v_C_v_T_bound_1D} with~\eqref{eqn:fft_bound_1D}:
    \begin{equation} \label{eqn:diff_of_vecs_1D}
        \norm{\beta v_{\kappa}^T- v_{\kappa}^C}^2_F \leq 2\left[ 1 - \sqrt{1-\frac{S_B(\hrho)}{\delta_{\kappa}^{2}}} \right].
    \end{equation}
    Finally, substituting \eqref{eqn:diff_of_vecs_1D} back into \eqref{eqn:1D_error_bound_1st_part} completes our proof.
\end{proof}

\section{Calculations of Rotations in Fourier-Bessel Space} 
\label{proof:rotations_2D}

Proof of equation \ref{eqn:rotation_calc_2D}:
\begin{equation}
\begin{split}
    x(\theta-\varphi,r) &= \sum_{\left(k,q\right)\in I}\hx[k,q]u_{k,q}\left(\theta-\varphi,r\right)\\
    &=\sum_{\left(k,q\right)\in I}\hx[k,q] e^{\iota  k(\theta- \varphi)} J_{q}\left(r\right)\\
    &= \sum_{\left(k,q\right)\in I}\hx[k,q]e^{-\iota  k\varphi} e^{\iota  k\theta} J_{q}\left(r\right)\\
    &= \sum_{\left(k,q\right)\in I}\hx[k,q]e^{-\iota  k\varphi}u_{k,q}\left(\theta,r\right).
\end{split}
\end{equation}

In Fourier Bessel (F-B) space, the measurements $y_{j}\left(\theta,r\right)$ are of the general form:
\begin{equation} \label{eqn:model_expanded_2D}
\begin{split}
y_{j}\left(\theta,r\right)&=\sum_{\left(k,q\right)\in I}\hx[k,q]e^{\iota k\left(\theta-\varphi_{j}\right)}J_{q}\left(r\right)+\varepsilon_{j}\left(\theta,r\right)\\
    &=\sum_{\left(k,q\right)\in I}\hx[k,q]e^{-\iota k\varphi_{j}}e^{\iota k\theta}J_{q}\left(r\right)+\varepsilon_{j}\left(\theta,r\right) \\
    &=\sum_{\left(k,q\right)\in I}\hx[k,q]e^{-\iota  k\varphi_j}u_{k,q}\left(\theta,r\right)\\
    &+
    \sum_{(k,q)\in\Z\times\Z_{\geq 0}}\hat{\varepsilon}_j[k,q] u_{k,q}\left(\theta,r\right) \\
    &= \sum_{(k,q)\in I}\left[\hx[k,q]e^{-\iota k\varphi_j}+\hat{\varepsilon}_j[k,q]\right] u_{k,q}\left(\theta,r\right)+\\
    &+\sum_{(k,q)\in\Z\times\Z_{\geq 0}\backslash I}\hat{\varepsilon}_j[k,q] u_{k,q}\left(\theta,r\right),
\end{split}
\end{equation}
where $\varepsilon_j(\theta,r)=\sum_{(k,q)\in\Z\times\Z_{\geq 0}}\hat{\varepsilon}_j[k,q] u_{k,q}\left(\theta,r\right)$ is the F-B expansion of the noise $\varepsilon_j$.

\section{Proof of Lemma~\ref{lemma:Moment_Eq_2D}} \label{proof:Moment_Eq_2D}

\begin{proof}
    Using the fact that $\hx^{(j)}[k,q] = \hx[k,q]e^{-\iota k\varphi_j}+\hat{\varepsilon}_j[k,q]$ for $(k,q)\in I$:
    
    \begin{equation}
    \begin{split}
        &M_1[k,q]=\\
        &\E\left[\hx[k,q]e^{-\iota k\varphi}+\hat{\varepsilon}[k,q]\right] =\\
        &\hx[k,q]\E\left[e^{-\iota k \varphi}\right]+\E\left[\hat{\varepsilon}[k,q]\right]= \hx[k,q]\E\left[e^{-\iota k \varphi}\right].
    \end{split}
    \end{equation}
    
    \begin{equation}
    \begin{split}
        &M_{2}\left[k_{1},q_{1},k_{2},q_{2}\right]=   \\
        &\scalebox{.91}{$\mathbb{E}\left[\left(\hx[k_{1},q_{1}]e^{-\iota k_{1}\varphi}+\hat{\varepsilon}[k_{1},q_{1}]\right)\left(\hx[k_{2},q_{2}]e^{-\iota k_{2}\varphi}+\hat{\varepsilon}[k_{2},q_{2}]\right)^{*}\right]$}= \\
        &\scalebox{.98}{$\mathbb{E}\left[\hx[k_{1},q_{1}]\hx[k_{2},q_{2}]^{*}e^{-\iota k_{1}\varphi}e^{\iota k_{2}\varphi}\right]+\mathbb{E}\left[\hat{\varepsilon}[k_{1},q_{1}]\hat{\varepsilon}[k_{2},q_{2}]^{*}\right] $}=\\
        & \hx[k_{1},q_{1}]\hx[k_{2},q_{2}]^{*}\mathbb{E}\left[e^{-\iota\left(k_{1}-k_{2}\right)\varphi}\right]+ \sigma^{2}1[(k_1, q_1) = (k_2, q_2)].\\
    \end{split}
    \end{equation}
    Applying (\ref{fourier_of_rho}) concludes the proof.
\end{proof}

\section{On the 2-D Frequency Marching Algorithm}

First, we introduce Algorithm~\ref{alg:FM_2D}, which is the 2-D analog of Algorithm~\ref{alg:FM_1D}.
\begin{algorithm} \caption{ \label{alg:FM_2D}
 A frequency marching algorithm for the 2-D model}
    \begin{algorithmic}
     \item \textbf{Input:} $M_{1,\est}$, $M_{2,\est}$, and $\sigma$
     \item \textbf{Output:} $\hx_{\est}$, $\hrho_{\est}$
    \begin{enumerate}
        \item $M_{2,\est} \gets M_{2,\est} -  \sigma^2 \I_{\abs{I}}$ (debiasing)
        \item $\mathcal{S} \gets 2\pi D_{M_1}^{-1} M_{2,\est} D_{M_1^{*}}^{-1}$
        \For{$-2B\leq k_1, k_2 \leq 2B$}
        \item $S[k_1,k_2] \gets \mathcal{S}[k_1,q_1 = 0,k_2,q_2=0]$ \label{step:FM_2D_reduction_over_q}
        \EndFor
        \item $\hrho_{\est}[0] \gets \frac{1}{2 \pi}$
        \item $\hrho_{\est}[1] \gets \sqrt{\nicefrac{1}{2 \pi S[1,1]}}$
        \For{$2 \leq k \leq B$}
        \item $\hrho_{\est}[k] \gets \label{step:2kB_FM_2D}\frac{\hrho_{\est}[1]}{S[k,k-1]\hrho_{\est}^*[k-1]} $
        \EndFor
        \For {$B+1 \leq k \leq 2B$}
        \item \label{step:Bk2B_FM_2D}
        $\hrho_{\est}[k] \gets S[k-B,-B] \hrho_{\est}[k-B] \hrho_{\est}[B]$
        \EndFor
        \For{$-2B \leq k \leq -1$}
        \item
        $\hrho_{\est}[k] = \hrho_{\est}^*[-k]$
        \EndFor
        \For{$-B \leq k \leq B$, $0 \leq q \leq Q_k-1$}
        \item $\mathcal{R}_{\est}[k, q] = \hrho[k]$
        \EndFor
        \item $\hx_{\est} \gets \frac{M_1}{2\pi\mathcal{R}_{\est}}$
    \end{enumerate}
    \end{algorithmic}
\end{algorithm}

\begin{proof}[Proof of Proposition~\ref{prop:FM_2D}] \label{proof:FM_2D}
    The first steps are identical to the 1-D case.
    De-biasing is performed, after which we use the property of $D_{\hx}=\frac{1}{2\pi}D_{M_1}D_{\mathcal{R}}^{-1}$, resulting in the following:
    \begin{equation}
        M_{2}=\frac{1}{2\pi}D_{M_1}D_{\mathcal{R}}^{-1} \mathcal{T}_{\rho} D_{\mathcal{R}^{*}}^{-1}D_{M_1^{*}}.
    \end{equation}
    Now, define
    \begin{equation}
        \mathcal{S} \triangleq 2\pi D_{M_1}^{-1} M_{2} D_{M_1^{*}}^{-1}=D_{\mathcal{R}}^{-1} \mathcal{T}_{\rho} D_{\mathcal{R}^{*}}^{-1}.
    \end{equation}
 \(\mathcal{S}\) has a structure similar to that of (\ref{S_def}), with its elements repeated to form constant blocks:
    \begin{equation}
        \begin{split}
            (\mathcal{S})_{k_1,k_2} &= \frac{\hrho[k_1-k_2]}{\hrho[k_1] \hrho^*[k_2]} \mathds{1}_{Q_{\abs{k_1}}\times Q_{\abs{k_2}}} \\&= S[k_1, k_2] \mathds{1}_{Q_{\abs{k_1}}\times Q_{\abs{k_2}}}, \: -B \leq k_1,k_2 \leq B.
        \end{split}
    \end{equation}
    In step~\ref{step:FM_2D_reduction_over_q}, $\mathcal{S}$ is reduced over the radial dimension:
    \begin{equation}
        S[k_1,k_2] = \mathcal{S}[k_1,q_1 = 0,k_2,q_2=0]
        = \frac{\hrho[k_1-k_2]}{\hrho[k_1] \hrho^*[k_2]}.
    \end{equation}
    From here on, subsequent steps coincide with the application of Algorithm~\ref{alg:FM_1D} (steps~\ref{step:FM_1D_rho_0} through \ref{step:FM_1D_conjugate_fill}) on the resulting $S$, from which $\hrho$ is recovered. $\mathcal{R}$ is then constructed, and $\hx$ is recovered from $M_1$ using $\hx=\frac{M_1}{2\pi \mathcal{R}}$.
\end{proof}

\subsection{A robust 2-D frequency marching} \label{subsec:better_FM_2D}

The 2-D Frequency Marching algorithm follows a similar idea to its 1-D counterpart, but differs in the fact that the radial dimension is reduced as only elements of $M_2$ where $q=0$ are used for $S$.

We propose a different method of reduction, which is comprised of the application of an average over the redundant $q$'s. In particular, 
\begin{enumerate}
    \item 
    Define $\left\{\omega^{k_1,k_2}_{q_1,q_2}\right\}_{\substack{0 \leq q_1 \leq Q_{\abs{k_1}}\\ 0 \leq q_2 \leq Q_{\abs{k_2}}}}$ s.t. $\sum_{q_1,q_2} \omega^{k_1,k_2}_{q_1,q_2} = 1$. 
    \item 
    Replace step~\ref{step:FM_2D_reduction_over_q} with
\begin{equation}
    S[k_1, k_2] \gets \sum_{q_1,q_2} \omega^{k_1,k_2}_{q_1,q_2} \mathcal{S}[k_1,q_1,k_2,q_2].
\end{equation}
\end{enumerate}
Then, all other modifications discussed in Appendix~\ref{subsec:better_FM_1D} can be applied to the resulting matrix $S$.

\section{The 2-D Spectral Algorithm}
In this section, we present a 2-D spectral algorithm that extends Algorithm~\ref{alg:spectral_algorithm_1D} in a natural way from the 1-D setting. While the core ideas remain analogous, the 2-D case introduces several unique nuances, primarily due to the distinctive structure of $\mathcal{T}_{\rho}$, a block Toeplitz matrix whose blocks of varying sizes $Q_{\abs{k_1}} \times Q_{\abs{k_2}}$, each mirror the matrix $T_{\rho}$ from the 1-D formulation.

For simplicity, we first consider the case where
\begin{equation} \label{eqn:Q_equal}
	Q_k=Q , \quad  0 \leq k \leq B,
\end{equation}
and $|I| = Q(2B+1)$. Namely, each angular frequency has a similar radial bandwidth, which may be restrictive for natural images. All following results will assume \eqref{eqn:Q_equal}. 

We present the 2-D procedure in Algorithm~\ref{alg:spectral_algorithm_2D}.

\begin{algorithm} 
\caption{A Spectral Algorithm for the 2-D model}  \label{alg:spectral_algorithm_2D}
    \begin{algorithmic}
        \item \textbf{Input:} $M_1$, $M_2$, $\sigma$
        \item \textbf{Output:} $\hx_{\est}$, $\hrho_{\est}$
        \begin{enumerate}
            \item $M_2 \gets M_2 - \sigma^2 \I$ (debiasing) \label{step:debias_2D}
            \item $P_{\hx} \gets \diag(M_2)$
            \item $ \mathcal{M}_2 \gets D_{\frac{1}{\sqrt{P_{\hx}}}} M_2 D_{\frac{1}{\sqrt{P_{\hx}}}}$ \label{step:2D_mathcal_M_2}
            \item Find the eigenvalue decomposition of $\mathcal{M}_2$: non-zero eigenvalues $\lambda_0 \geq \lambda_1 \geq \ldots \geq \lambda_{K}$ with corresponding normalized eigenvectors $v_0, v_1, \ldots, v_{K}$. \label{step:mathcal_m2_eigenvalue_decomp_2D}
            \item $\kappa \gets \argmax_{0 \leq k \leq K} \min_{k' \neq k} \abs{\lambda_{k'}-\lambda_{k}}$ \label{step:unique_eigenvector_2D}
            \item $\tx_{\est} \gets \sqrt{(2B+1)Q}v_{\kappa}$ \label{step:fft_2D}
            \item $\beta \gets e^{\iota (\measuredangle M_1[k=0, q=0] - \measuredangle \tx_{\est}[k=0, q=0])}$ \label{step: phase_correction_2D}
            \item $\tx_{\est} \gets \beta \tx_{\est}$ \label{step:def_x_tilde_2D}
            \item $\hx_{\est} \gets  \sqrt{P_{\hx}}\odot \tx_{\est}$ \label{step:power_spectrum_back_2D}
            \item $\hrho_{\est} \gets \frac{M_1}{2 \pi \hx}$
        \end{enumerate}
    \end{algorithmic}
\end{algorithm}

Next, we explain some key steps in Algorithm~\ref{alg:spectral_algorithm_2D}. We start with updating our notation. 
Consider a block circulant matrix, that is, a matrix where each block is circulant. This is our generalization of a circulant matrix for the 2-D case. 

Let $A \in \C^{Q \times (2B+1)}$ be a matrix. We define $\vecb(A) \in \C^{Q(2B+1)\times 1}$ to be its vectorized version where $A$'s columns are stacked on top of each other. In addition, we define $\mat$ to be the inverse operator of $\vecb$, where $\mat$ splits a $Q(2B+1)\times 1$ column vector to a $Q \times (2B+1)$ matrix. Next, we define a $\BC$ (block - circulant) matrix to be a matrix of size $Q(2B+1)\times Q(2B+1)$ which is made of a $(2B+1) \times (2B+1)$ configuration of blocks with size $Q \times Q$, where each block is a circulant matrix, and the $(2B+1) \times (2B+1)$ formation is also circulant with respect to its block entries. We denote a block-circulant matrix with first column $\vecb(V)$ as $\BC(V)$. 

Let $W_{2B+1}\in \C^{(2B+1)^2}$ and $W_Q\in \C^{Q^2}$ be the normalized DFT matrices of their respective sizes, and let $\otimes$ be the Kronecker product. Define $W=W_{2B+1} \otimes W_Q$ to be the Kronecker product between $W_{2B+1}$ and $W_Q$. For any $v\in\C^{Q(2B+1)\times 1}$ and $V\in \C^{Q \times (2B+1)}$ we have:
\begin{align}
    \BC(V)&=W^{*} \diag(\vecb(\fftt(V))) W\\
    D_{v} &= W \BC(\ifftt(\mat(v))) W^*,
\end{align}
where $\fftt$ represents the 2-D DFT.

We return to Algorithm~\ref{alg:spectral_algorithm_2D}, regarding step~\ref{step:2D_mathcal_M_2} we have,
    \begin{equation} \label{eqn:2D_eigen_decomp}
    \begin{split}
        \mathcal{M}_2 &\triangleq D_{\frac{1}{\sqrt{P_{\hx}}}} M_2 D_{\frac{1}{\sqrt{P_{\hx}}}} \\
        &= 2\pi  D_{\frac{1}{\sqrt{P_{\hx}}}} \left(D_{\hx}\BC(\Vopt)D_{\hx^*}\right) D_{\frac{1}{\sqrt{P_{\hx}}}} \\
        &= 2\pi D_{\tx}\BC(\Vopt) D_{\tx}^{*} \\
        &= 2\pi W\mathcal{A} W^{*} \BC(\Vopt) W \mathcal{A}^* W^* .
    \end{split}\end{equation}Where we denote $\mathcal{A}=\BC(\ifftt(\mat(\tx)))$ for the sake of brevity. $\mathcal{A}$ is in fact a unitary matrix:\begin{equation}
    \begin{split}
        \mathcal{A}\mathcal{A}^{*} &= W^{*} W \mathcal{A} W^{*} W \mathcal{A}^{*} W^{*} W\\
        &= W^{*} \left(W \mathcal{A} W^{*}\right) \left(W \mathcal{A}^{*} W^{*}\right) W\\
        &= W^{*} (D_{\tx} D_{\tx}^{*}) W\\
        &= W^{*} \I W = \I.
    \end{split}
    \end{equation}
    Note that 
    \begin{equation}
         (W^{*} \BC(\Vopt) W)^t = W \BC(\tVopt) W^*,
    \end{equation}
        where $\tVopt \in \C^{Q\times (2B+1)}$  just consists of the vector $\tvopt^t$ (defined in equation~\eqref{eqn:v_tilde_opt}) duplicated over all $Q$ rows. Therefore,
    \begin{equation}
    \fftt(\tVopt) = \fftt\begin{bmatrix}
    \tvopt^t\\
    \tvopt^t\\
    \vdots\\
    \tvopt^t
    \end{bmatrix}
    = Q \begin{bmatrix}
    \fft(\tvopt^t)\\
    0_{Q-1 \times (2B+1)},
    \end{bmatrix}
    \end{equation}
    which implies that 
    \begin{equation} \label{eqn:diagonalization_BCCB_2D}
    \begin{split}
    W \BC(\tVopt) W^* &= \diag(\vecb(\fftt(\tVopt)))\\
    &= Q\diag \left[\vecb\begin{pmatrix}
    \fft(\tvopt^t)\\
    0_{Q-1 \times (2B+1)}
    \end{pmatrix}\right].
    \end{split}
    \end{equation}
    
    After substituting \eqref{eqn:diagonalization_BCCB_2D} into \eqref{eqn:2D_eigen_decomp}, we see that~\eqref{eqn:2D_eigen_decomp} constitutes a unitary eigenvalue decomposition. Note that $\mathcal{M}_2$ is of low rank (if $Q>1$), meaning that an isolated eigenvalue exists from the assumption on $\hvopt$. 
    
    In steps \ref{step:mathcal_m2_eigenvalue_decomp_2D} and \ref{step:unique_eigenvector_2D}, we pick an isolated eigenvector, which must appear as a column of $W\mathcal{A}$, up to a constant phase. This phase is set later in step \ref{step: phase_correction_2D} using the phase of $M_1[k=0, q=0] = x[k=0, q=0]$. 
    
    Step \ref{step:fft_2D} salvages the correct vector of phases $\tx_{\est}$, up to some global rotation. 
    It is clear to see from (\ref{eqn:diagonalization_BCCB_2D}) that for some $0 \leq l_1 \leq Q-1$ and $0 \leq l_2 \leq 2B$, the eigenvalue corresponding to the $(l_1+l_2Q)$'th column of $W\mathcal{A}$ is $0$ if $l_1\neq 0$. This property guarantees that an eigenvector of $\mathcal{M}_2$ with a non-degenerate, non-zero eigenvalue is
    \begin{equation} \label{eqn:fft2_0_l2}
        \begin{split}
        v_{\kappa} &= \alpha W \vecb\left((0,l_2)\circ \ifftt(\mat(\tx))\right) \\
        &= \frac{\alpha}{\sqrt{(2B+1)Q}} \fftt\left((0,l_2)\circ \ifftt(\mat(\tx))\right).
        \end{split}
    \end{equation}
    Here $\abs{\alpha} = 1$. Similarly to the 1-D case in \eqref{eqn:fft1_l_0_2B}, $v_{\kappa}$ is equivalent to $\tx$ up to a factor of $\sqrt{(2B+1)Q}$, multiplied by a phase which rotates the image by $\varphi = \frac{2\pi}{2B+1}l_2$, up to an arbitrary constant phase.
    
    In Step~\ref{step:power_spectrum_back_2D}, the power spectrum is introduced back. Finally, the distribution is recovered through $M_1$.
    
The following proves the correctness of the algorithm.
\begin{prop}
\label{prop:spectral_algorithm_perfect_2D}
   Assume that $S_B(\hrho)=0$, $\hx$ is non-vanishing and $\hvopt$ has a non-zero isolated value. Then, given exact $M_1$, $M_2$ and $\sigma$, Algorithm~\ref{alg:spectral_algorithm_2D} exactly recovers $\hx$ and the first $B$ Fourier coefficients of $\rho$, up to a global rotation. 
\end{prop}

As per the 1-D case, the assumption $S_B(\hrho) = 0$ is not always met. Therefore, we want to derive an analog for Theorem~\ref{thm:spectral_bound_1D} in the 2-D case.

Since we wish to approximate $\mathcal{T}_{\rho}$ with a $\BC$ matrix, we shall consider the following minimization problem:
\begin{equation}
    \min_{V \in \C^{Q \times 2B+1}} \norm{\mathcal{T}_{\rho}-\BC(V)}_F^2.
\end{equation}
Note that $\mathcal{T}_{\rho}$ consists of constant blocks, 
where the minimum is achieved by $\Vopt \in \C^{Q\times(2B+1)}$ which consists of the vector $\vopt^t$ of~\eqref{eqn:v_opt} repeated over the $Q$ rows. 
The minimum is 
\begin{equation}
    \norm{\mathcal{T}_{\rho}-\BC(\Vopt)}_F^2 = Q^2 \norm{T-C_{\vopt}}_F^2 = Q^2 S_B(\hrho),
\end{equation}

We are ready for the 2-D version of Theorem~\ref{thm:spectral_bound_1D}.
\begin{theorem} \label{thm:spectral_bound_2D}
   Consider $\lambda^{T}_0 \geq \lambda^{T}_1 \geq \ldots \geq \lambda^{T}_{K_T}$ and $\lambda^{C}_0 \geq \lambda^{C}_1 \geq \ldots \geq \lambda^{C}_{K_C}$ to be the non-zero eigenvalues of $\mathcal{T}_{\rho}$ and $\BC(\Vopt)$ respectively. We write $P_{\max}:=\max_{(k,q)\in I}P_{\hx}[k, q]$ and $\delta_{\kappa}=\displaystyle \max\left(\min_{j \neq \kappa}{\abs{\lambda^{C}_{\kappa}-\lambda^{T}_{j}}}, \min_{j \neq \kappa}{\abs{\lambda^{C}_{j}-\lambda^{T}_{\kappa}}}\right)$.
   Assume the following conditions hold:
   \begin{enumerate}
       \item 
       $\hx$ is non-vanishing.
       \item 
       $\kappa$ of Step~\ref{step:unique_eigenvector_2D} of Algorithm~\ref{alg:spectral_algorithm_2D} has both $\lambda_{\kappa}^{T}$ and $\lambda_{\kappa}^{C}$ with an eigenspace of dimension 1.
       \item 
       The estimate calculated in Step~\ref{step:def_x_tilde_2D}  of Algorithm~\ref{alg:spectral_algorithm_2D} satisfies $\tx_{\est}^*[\Psi_{\frac{2\pi l}{2B+1}} \odot\tx] \geq 0$.
       \item 
       $Q^2 S_B(\hrho) \leq \delta_{\kappa}^2$.
   \end{enumerate} 
   Then, there exists $ 0 \leq l \leq 2B$, such that
   \begin{equation} \label{eqn:2D_error_bound_equation}
       \resizebox{0.95\linewidth}{!}{$\norm{\hx_{\est} - \Psi_{\frac{2\pi l}{2B+1}} \odot \hx}^2_F \leq 
       2Q(2B+1) \displaystyle P_{\max} \Big[1-\sqrt{1-\frac{Q^2 S_B(\hrho)} {\delta_{\kappa}^2}}\Big]$},
    \end{equation}
    where $\Psi_{\varphi}[k,q] = e^{- \iota k \varphi}$. 
\end{theorem}

\begin{proof}
    As per the proof of theorem \ref{thm:spectral_bound_1D}, we define $\mathcal{M}_2^C$ to be the matrix $\mathcal{M}_2$ would have been if $\mathcal{T}_{\rho}=\BC(\Vopt)$, meaning that 
    \begin{equation}
        \mathcal{M}_{2}^{C}\triangleq 2\pi W \mathcal{A}W^{*} \BC(\Vopt) W \mathcal{A}^{*} W^*.
    \end{equation}
    Both $\mathcal{M}_{2}^{C}$ and $\mathcal{M}_{2}$ differ only by conjugation by \emph{the same} unitary matrix. Therefore,
    \begin{equation} \label{eqn:fro_eq_2D}
        \norm{\mathcal{M}_2^C - \mathcal{M}_2}_F^2 = (2\pi)^2 \norm{\mathcal{T}_{\rho}-\BC(\Vopt)}_F^2 = (2 \pi)^2 Q^2 S_B(\hrho) .
    \end{equation}
    The eigenvalues of $\mathcal{M}_2$, $\mathcal{M}_2^C$ are equivalent to those of $\mathcal{T}_{\rho}$, $\BC(\Vopt)$, multiplied by  $2\pi$, respectively. Thus, we have $$2\pi\lambda^{T}_0 \geq 2\pi\lambda^{T}_1 \geq \ldots \geq 2\pi\lambda^{T}_{K_T} ,$$ 
    and 
    $$2\pi\lambda^{C}_0 \geq 2\pi\lambda^{C}_1 \geq \ldots \geq 2\pi\lambda^{C}_{K_C},$$ 
    as the non-zero eigenvalues of $\mathcal{M}_2$ and $\mathcal{M}_2^C$, respectively. Their corresponding normalized eigenvectors are $v_0^T,\ldots,v_{K_T}^T$ and $v_0^C,\ldots,v_{K_C}^C$, where $0 \leq K_T, K_C \leq 2B$. 
    
    Similar to Appendix~\ref{proof:spectral_bound_1D}, we simplify~\eqref{eqn:fft2_0_l2} to have
    \begin{equation}
        v_{\kappa}^C = \frac{1}{\sqrt{(2B+1)Q}} \Psi_{\frac{2 \pi l}{2B+1}} \odot \tx , 
    \end{equation}
    as the sole normalized eigenvector for $2 \pi\lambda_{\kappa}^C$, up to some global phase.
    Therefore, 
    \begin{equation} \label{eqn:vk_c_normilized}
        \sqrt{P_{\hx}} \odot v^C_{\kappa} = \frac{1}{\sqrt{(2B+1)Q}} \Psi_{\frac{2 \pi l}{2B+1}} \odot \hx.
    \end{equation}

    Now, consider the relation between $\hx_{\est}$ and (\ref{eqn:vk_c_normilized}):
    \begin{equation} \label{eqn:2D_error_bound_1st_part}
    \begin{split}
        &\norm{\hx_{\est} - \Psi_{\frac{2 \pi l}{2B+1}} \odot \hx}^2_F \\
         = &
         \norm{\sqrt{P_{\hx}} \odot \left[ \tx_{\est} -  \sqrt{(2B+1)Q} v_{\kappa}^C\right]}^2_F\\
         \leq & \left[\max_{(k,q)\in I} P_{\hx}[k,q]\right] \norm{\tx_{\est} - \sqrt{(2B+1)Q} v_{\kappa}^C}^2_F\\
          = & (2B+1) Q \left[\max_{(k,q)\in I} P_{\hx}[k,q]\right] \norm{\beta v_{\kappa}^T- v_{\kappa}^C}^2_F.
    \end{split}
    \end{equation}
    We next state the following lemma:
    \begin{lemma} \label{a_est_conj_flip_2D}
        $\tx_{\est}$ has the following property:
        \begin{equation}
            \tx_{\est}[k, q] = \tx_{\est}[-k, q]^*.
        \end{equation}
    \end{lemma}
    \begin{proof}
        Let $v$ be some eigenvector of $\mathcal{M}_2 = 2\pi D_{\tx}\mathcal{T}_{\rho}D_{\tx}^{*}$ with eigenvalue $\lambda$. Let $v'$ be a vector s.t. $v'[k, q] = v^*[-k, q]$, i.e., $v'$ is the conjugate of the vector $v$, where each angular frequency is flipped.
        \begin{equation}
        \begin{split}
            &\left[2\pi D_{\tx}\mathcal{T}_{\rho}D_{\tx}^{*} v'\right][k_1,q_1]= \\
            & \sum_{k_2=-B}^{B} \sum_{q_2=0}^{Q-1} 2\pi \tx[k_1,q_1] \tx[k_2,q_2]^{*} \hrho[k_1-k_2] \left[v[-k_2,q_2]\right]^*=\\
            & \sum_{k_2'=-B}^{B} \sum_{q_2=0}^{Q-1} 2\pi \tx[k_1,q_1] \tx[-k_2',q_2]^{*} \hrho[k_1+k_2'] \left[v[k_2',q_2]\right]^* =\\
            & \scalebox{0.87}{$\left[\sum_{k_2'=-B}^{B} \sum_{q_2=0}^{Q-1} 2 \pi  \tx[-k_1,q_1] \tx[k_2',q_2]^{*} \hrho[-k_1-k_2'] \left[v[k_2',q_2]\right]\right]^{*}$} =\\
            &\left[\lambda v[-k_1,q_1]\right]^{*} = \lambda v'[k_1,q_1].\\
        \end{split}
        \end{equation}
        Thus $v'$ is also an eigenvector of $\mathcal{M}_2$, with eigenvalue $\lambda$. Since $\tx_{\est}$ is an eigenvector of $\mathcal{M}_2$ corresponding to a non-degenerate eigenvalue, it must be that:
        \begin{equation}
            \tx_{\est}[k, q] = \alpha \tx_{\est}^*[-k, q],
        \end{equation}
        for some $\alpha\in\C$. The realness of $\tx_{\est}[0,0]\neq 0$ forces $\alpha = 1$, which completes our proof.
    \end{proof}
    \begin{coro} \label{conj_flip_eq_2D}
        Let $u,v\in\C^{ (2B+1)Q}$. If both $u[k, q] = u^*[-k, q]$ and $v[k, q] = v^*[-k, q]$,  then $u^*v\in\R$.
    \end{coro}
    \begin{proof}
        Under this assumption, we define the generating functions
        \begin{equation}
        \begin{split}
            f_U(\theta,r)=\sum_{(k,q)\in I} u[k,q] u_{k,q}(\theta,r)\\
            f_V(\theta,r)=\sum_{(k,q)\in I} v[k,q] u_{k,q}(\theta,r),
        \end{split}
        \end{equation}
        which the assumption insures are real functions.
        Using Parseval's theorem we get:
        \begin{equation}
            u^*v = \int_{D^2} f_u(\theta,r) f_v(\theta,r)drd\theta \in \R,
        \end{equation}
        where $D^2$ is the 2-dimensional unit disk.
    \end{proof}
    Following~\eqref{eqn:2D_error_bound_1st_part}, we note that $v_{\kappa}^C[k,q] = (v_{\kappa}^C[-k,q])^*$. We thus get:
    \begin{equation} \label{eqn:normdif}
    \begin{split}
       &\norm{\beta v_{\kappa}^T- v_{\kappa}^C}^2_F =\\
       &\norm{\beta v_{\kappa}^T}^2_F + \norm{v_{\kappa}^C}^2_F - 2\Re\left[(\beta v_{\kappa}^T)^* v_{\kappa}^C\right] = \\
       &2 -2\Re\left[(\beta v_{\kappa}^T)^* v_{\kappa}^C\right] \stackrel{(**)}{=} 2 - 2(\beta v_{\kappa}^T)^* v_{\kappa}^C,
    \end{split}
    \end{equation}
    where at (**) of \eqref{eqn:normdif}, we employ Lemma~\ref{a_est_conj_flip_2D} together with Corollary~\ref{conj_flip_eq_2D}. As in the 1-D case, we use the condition on the eigenvectors, meaning that $(\beta v_{\kappa}^T)^* v_{\kappa}^C$ is not only real but non-negative. Then,
    \begin{equation} \label{fft_bound_2D}
        \norm{\beta v_{\kappa}^T- v_{\kappa}^C}^2_F =2\left[ 1-\abs{v_{\kappa}^{T*} v_{\kappa}^C} \right].
    \end{equation}
    Now we are only left with finding a lower bound on $\abs{v_{\kappa}^{T*} v_{\kappa}^C}$, and as in the 1-D case, we will use the Davis-Kahan Theorem, and its conclusion stated in Equation (\ref{eqn:v_C_v_T_bound_1D}), which we then substitute in (\ref{fft_bound_2D}):
    \begin{equation} \label{eqn:v_C_v_T_bound_2D}
        \norm{\beta v_{\kappa}^T- v_{\kappa}^C}^2_F \leq 2(2B+1)Q\left[ 1 - \sqrt{1-\frac{Q^2 S_B(\hrho)}{\delta_{\kappa}^{2}}} \right].
    \end{equation}
    Finally, the proof follows using~\eqref{eqn:v_C_v_T_bound_2D} and~(\ref{eqn:2D_error_bound_1st_part}).
\end{proof}

\ifCLASSOPTIONcaptionsoff
  \newpage
\fi




\bibliographystyle{IEEEtran}


\end{document}